\documentclass[11pt]{article}
\usepackage{amsmath, amsthm}
\usepackage{latexsym}
\usepackage{graphicx, psfrag}
\makeatletter
\newfont{\footsc}{cmcsc10 at 8truept}
\newfont{\footbf}{cmbx10 at 8truept}
\newfont{\footrm}{cmr10 at 10truept}
\makeatother
\newtheorem{theorem}{\bf Theorem}
\newtheorem{proposition}{\bf Proposition}
\newtheorem{lemma}{\bf Lemma}

\newtheorem{corollary}{\bf Corollary}

\begin{document}
\title{The Shape of the Noncentral $\chi^2$ Density}

\author{Yaming Yu\\
\small Department of Statistics\\[-0.8ex]
\small University of California\\[-0.8ex] 
\small Irvine, CA 92697, USA\\[-0.8ex]
\small \texttt{yamingy@uci.edu}}

\date{}
\maketitle

\begin{abstract}
A noncentral $\chi^2$ density is log-concave if the degree of freedom is $\nu\geq 2$.  We complement this known result by showing that, for each $0<\nu<2$, there exists $\lambda_\nu>0$ such that the $\chi^2$ with $\nu$ degrees of freedom and noncentrality parameter $\lambda$ has a decreasing density if $\lambda\leq \lambda_\nu$, and is bi-modal otherwise.  The critical $\lambda_\nu$ is characterized by an equation involving a ratio of modified Bessel functions.  When an interior mode exists we derive precise bounds on its location. 

{\bf Keywords:} distribution theory; log-concavity; log-convexity; noncentral distribution; special functions; unimodality. 

{\bf MSC2010 Classifications:} 60E05; 62E15; 33C10.
\end{abstract}

\section{Introduction}
The central $\chi^2$ density with $\nu$ degrees of freedom is log-concave if $\nu\geq 2$ and log-convex otherwise.  The noncentral $\chi^2$ density with $\nu$ degrees of freedom and noncentrality parameter $\lambda$ is still log-concave if $\nu\geq 2$.  It is decreasing if $0< \nu <2$ and $\lambda$ is small.  When $0<\nu<2$ and $\lambda$ is large, however, a new shape emerges which some may find surprising: the noncentral $\chi^2$ density can be bi-modal! 

This note characterizes the noncentral $\chi^2$ density as either log-concave, decreasing, or bi-modal (the first two categories overlap when $\nu=2$) by delineating the range of parameters for each.  Our contribution is a mathematical derivation of some facts about the noncentral $\chi^2$ that are ``folklore,'' i.e., either tacitly assumed or supported by numerical evidence.  This happens to be an interesting exercise in special functions.  Part of the criteria is expressed in terms of a modified Bessel function equation, and the method may yield further results (e.g., inequalities) concerning  modified Bessel functions. 

\section{Definitions}
The central $\chi^2$ density with $\nu>0$ degrees of freedom is given by 
$$p_{\nu, 0}(x)=\frac{e^{-x/2}}{2\Gamma(\nu/2)} \left(\frac{x}{2}\right)^{\nu/2-1},\quad x>0.$$
The noncentral $\chi^2$ density with $\nu$ degrees of freedom and noncentrality parameter $\lambda$ is defined as 
\begin{equation}
\label{dens}
p_{\nu, \lambda}(x)=\sum_{k=0}^\infty \frac{e^{-\lambda/2}}{k!}\left(\frac{\lambda}{2}\right)^k p_{\nu+2k, 0}(x),\quad x>0.
\end{equation}
That is, $p_{\nu, \lambda}$ is a Poisson mixture of central $\chi^2$ densities.  When $\nu$ is a positive integer, the noncentral $\chi^2$ arises naturally as the distribution of quadratic forms in normal variables.  Specifically, suppose $X_1,\ldots, X_\nu$ are independent normal random variables with means $\mu_k,\ k=1,\ldots, \nu,$ and unit variances, then the density of $\sum_{k=1}^\nu X_k^2$ is precisely $p_{\nu, \lambda}$ given by (\ref{dens}) with $\lambda=\sum_{k=1}^\nu \mu_k^2$.  The definition (\ref{dens}), however, does not require $\nu$ to be an integer. 

Distributional properties of the noncentral $\chi^2$ can be difficult to obtain because the density is ``not in closed form''.  Another expression for $p_{\nu, \lambda}$, not necessarily more tractable, is 
\begin{equation}
\label{dens1}
p_{\nu, \lambda}(x)=\frac{1}{2} e^{-(x+\lambda)/2} \left(\frac{x}{\lambda}\right)^{(\nu-2)/4} I_{(\nu-2)/2}(\sqrt{\lambda x}).
\end{equation}
Here $I_\nu(x)$ denotes the modified Bessel function of the first kind given by  
$$I_\nu(x) = \sum_{k=0}^\infty \frac{(x/2)^{2k+\nu}}{k!\Gamma(\nu+k+1)}.$$
The ratio
\begin{equation}
\label{r}
r_\nu(x)=\frac{I_\nu(x)}{I_{\nu-1}(x)}
\end{equation}
plays a crucial role in our analysis.  We use $I'_\nu(x)$ and $r'_\nu(x)$ to denote derivatives with respect to $x$. 

We are concerned with shape properties of the noncentral $\chi^2$ such as unimodality, log-concavity and log-convexity.  A function $f(x)$ is unimodal if there exists $x_*$ such that $f(x)$ increases for $x\leq x_*$ and decreases for $x\geq x_*$ ($x_*$ is called a mode).  A nonnegative function $f$ is log-concave (respectively, log-convex) if $\log f$ is concave (respectively, convex).  It is well known that log-concavity implies unimodality.  Since we focus on probability densities supported on $(0, \infty)$, all statements concerning monotonicity and log-convexity refer to this interval. 

\section{Relevant Literature}
Johnson, Kotz and Balakrishnan (1995) discuss the noncentral $\chi^2$ distribution theory in great detail.  Saxena and Alam (1982) study estimation methods for the noncentrality parameter.  Ding (1992) and Kn\"{u}sel and Bablok (1996) propose methods to compute the density and distribution functions numerically.  Concerning distributional properties, Karlin (1968) has proved that if $\nu\geq 2$ then the noncentral density $p_{\nu, \lambda}$ is log-concave.  Das Gupta and Sarkar (1984) and Finner and Roters (1997) derive related results, including log-concavity of the distribution and survival functions and generalizations to noncentral $F$ and $t$ distributions; see also van Aubel and Gawronski (2003). 

Not only is the noncentral $\chi^2$ distribution important for power calculations in hypothesis testing, it also has engineering applications including radar detection and communications over fading channels.  It is intimately related to the generalized Marcum Q function $Q_\nu(a, b)$, which is of great interest in information theory and communications.  Specifically, $Q_\nu(a, b)$ is the survival function of a noncentral $\chi^2$ distribution with $2\nu$ degrees of freedom and noncentrality parameter $a^2$ evaluated at $b^2$: 
$$Q_\nu(a, b)= \int_{b^2}^\infty p_{2\nu, a^2}(x)\, {\rm d}x,\quad a,\, b\geq 0.$$ 
See, for example, Marcum (1960), Nuttall (1975), Shnidman (1989), Helstrom (1992), Ross (1999), Simon and Alouini (2003), Li and Kam (2006), Sun, Baricz, and Zhou (2010), and Yu (2011).  Sun {\it et al.} (2010) also derive log-concavity results.  Yu (2011) establishes the log-concavity of $Q_\nu(a, b)$ in $b$ for $\nu\geq 1/2$, confirming a conjecture of Sun {\it et al.} (2010).  This work uses methods that are similar to those of Yu (2011). 

\section{Main Results}
For $\nu\in (0, 2)$ define 
\begin{equation*}
g_\nu(\lambda)\equiv r_{\nu/2}(\sqrt{\lambda(\lambda+\nu-4)}) - \frac{\lambda-2}{\sqrt{\lambda(\lambda+\nu-4)}},\quad \lambda > 4-\nu,
\end{equation*}
where $r_\nu(x)$ is given by (\ref{r}).  We rely on a sign property of $g_\nu(\lambda)$. 
\begin{lemma}
\label{lem1}
For each $\nu\in (0, 2)$ there exists a unique $\lambda_\nu\in (4-\nu, \infty)$ such that $g_\nu(\lambda)> 0$ or $<0$ according as $\lambda >\lambda_\nu$ or $ <\lambda_\nu$.
\end{lemma}
In other words, as $\lambda$ increases, $g_\nu(\lambda)$ crosses zero exactly once from below.  The unique crossing point $\lambda_\nu$ can be computed by bisection or functional iteration.  A small sample is recorded in Table~1.  Proposition~\ref{prop1} recognizes the limiting cases. 
\begin{proposition}
\label{prop1}
We have 
$$\lim_{\nu\downarrow 0} \lambda_\nu = 4,\quad{\rm and}\quad \lim_{\nu\uparrow 2}\lambda_\nu = 2.$$ 
\end{proposition}

\begin{table}
\caption{Values of $\lambda_\nu$ for selected $\nu\in (0, 2)$.}
\begin{center}
\begin{tabular}{c|rrrrrrrrr}
\hline
$\nu$         & $0+$    & $0.25$  & $0.5$   & $0.75$  & $1$      & $1.25$  & $1.5$   & $1.75$  &  $2-$ \\
\hline
$\lambda_\nu$ & $4.000$ & $4.769$ & $4.661$ & $4.467$ & $4.217$  & $3.914$ & $3.548$ & $3.073$ & $2.000$ \\
\hline
\end{tabular}
\end{center}
\end{table}

Theorem~\ref{main}, our main result, says that a noncentral $\chi^2$ density is unimodal unless the degree of freedom is fewer than two and, at the same time, the noncentrality parameter is too large.  The $\lambda_\nu$ in Lemma~\ref{lem1} is the critical value that separates the decreasing from the bi-modal categories. 
\begin{theorem}
\label{main}
The noncentral $\chi^2$ density with $\nu>0$ degrees of freedom and noncentrality parameter $\lambda\geq 0$ is 
\begin{itemize}
\item[i]
log-concave iff $\nu\geq 2$; 
\item[ii]
strictly log-convex for small $x$ and strictly log-concave afterwards iff $0<\nu<2$ and $\lambda>0$; 
\item[iii]
decreasing iff $0<\nu\leq 2$ and $\lambda\leq \lambda_\nu$; 
\item[iv]
bi-modal iff $0< \nu < 2$ and $\lambda> \lambda_\nu$.
\end{itemize}
\end{theorem}

Part (i) of Theorem~\ref{main} is well known (see Karlin 1968 or Finner and Roters 1997).  But we are unaware of any previous work offering the precise distinction between (iii) and (iv).  Figure~1 displays some noncentral $\chi^2$ densities, including a bi-modal case ($\nu=1,\ \lambda=5$). 

An extreme case of bi-modality is the noncentral $\chi^2$ with zero degrees of freedom, which is a mixture of a continuous distribution and a point mass at zero.  The probability at zero is $e^{-\lambda/2}$, and the density of the continuous part is proportional to (\ref{dens}) upon replacing $\sum_{k=0}^\infty$ by $\sum_{k=1}^\infty$ and letting $\nu=0$.  See Siegel (1979) for a detailed exploration of this distribution.

\begin{figure}
\begin{center}
\includegraphics[width=3.7in, height=5in, angle=270]{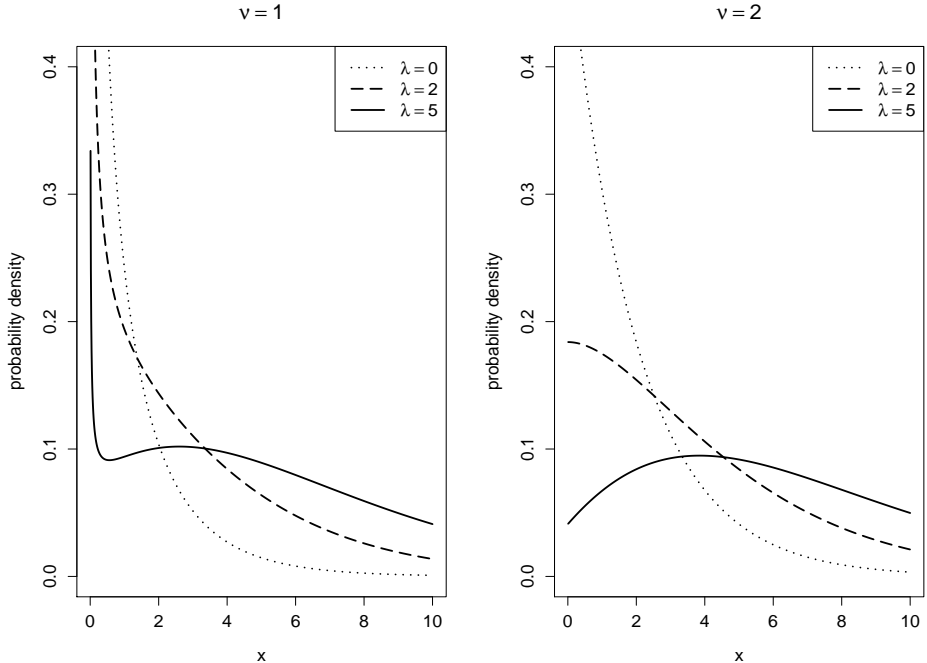}
\end{center}
\caption{Noncentral $\chi^2$ densities with $\nu$ degrees of freedom and noncentrality parameter $\lambda$.} 
\end{figure}

Next, we study the location of the modes and derive some simple but effective bounds.  Corollary~\ref{coro1} easily follows from Theorem~\ref{main} and its proof. 
\begin{corollary} 
\label{coro1}
The density $p_{\nu, \lambda}$ has 
\begin{itemize}
\item[i]
a unique mode in $(0, \infty)$ if $\nu > 2$ or $\nu=2$ and $\lambda> 2$;
\item[ii]
a unique mode at zero if $0<\nu\leq 2$ and $\lambda \leq \lambda_\nu$;
\item[iii] 
a mode at zero and another in $(0,\infty)$ if $0<\nu <2$ and $\lambda >\lambda_\nu$.  
\end{itemize}
\end{corollary}
Part (i) of Corollary~\ref{coro1} also appears in Corollary 3.3, part (i) of van Aubel and Gawronski (2003).  
Corollary 1 implies that, if $p_{\nu, \lambda}$ has an interior mode, then so does $p_{\nu, \tilde{\lambda}}$ for 
any $\tilde{\lambda}>\lambda$.  Moreover, according to Proposition~\ref{monomode}, the mode moves to the right as $\lambda$ 
increases.  This is a known result in the log-concave case. 

\begin{proposition}
\label{monomode}
If $\nu>2$ or $0<\nu\leq 2$ and $\lambda>\lambda_\nu$ then an interior mode $M(\nu, \lambda)$ exists and, for fixed $\nu$, 
$M(\nu, \lambda)$ increases with $\lambda$. 
\end{proposition} 
It is known that $M(\nu, \lambda)$ increases in $\nu\in (2, \infty)$ for fixed $\lambda\geq 0$ (van Aubel and Gawronski 2003).  Such monotonicity with respect to $\nu$ does not seem to extend to the bi-modal case.  In fact, inspection of Table~1 reveals that, for certain values of $\lambda$, the interior mode may disappear and then reappear as $\nu$ increases in 
$(0, 2)$. 

When an interior mode exists, Proposition~\ref{prop2} gives more information about its location. 
\begin{proposition}
\label{prop2}
If an interior mode $M(\nu, \lambda)$ exists then it 
satisfies
\begin{equation}
\label{loose}
\lambda+\nu-4 < M(\nu, \lambda).
\end{equation}
Moreover, if $\nu\geq 2$ then 
\begin{align}
\label{bound1}
(\nu-2)\left(1+\frac{\lambda}{\nu}\right) \leq M(\nu, \lambda) \leq \lambda + \nu -2.
\end{align}
If $\nu>3$ then 
\begin{equation}
\label{bound2}
\lambda +\nu -3 < M(\nu, \lambda). 
\end{equation}
If $0<\nu<2$ and $\lambda>\lambda_\nu$ then 
\begin{equation}
\label{bound3}
M(\nu, \lambda) < \lambda +\nu -3. 
\end{equation}
\end{proposition}
Sen (1989) derives (\ref{bound1}) and (\ref{loose}) for $\nu\geq 2$.  Van Aubel and Gawronski (2003, Theorem 6.1, part (i)) present (\ref{bound2}) for $\nu > 4$. The bounds (\ref{loose}) and (\ref{bound3}) extend existing results to the bi-modal case.  The two bounds in (\ref{bound1}) become tight as $\lambda\downarrow 0$ and in fact coincide for $\lambda=0$.  The bounds (\ref{bound2}) and (\ref{bound3}) are remarkable in view of the asymptotic formula 
$$M(\nu, \lambda)=\lambda+\nu-3+O(\lambda^{-1}),\quad {\rm as}\quad \lambda\to \infty.$$
We emphasize that $M(\nu, \lambda)$ refers to the interior (nonzero) mode. 

In Section~5 we prove the main results.  The method is elementary although it does rely on properties of modified Bessel functions.  Part of the argument may be helpful in establishing inequalities in other contexts.  Of course, it would be interesting to derive similar results for noncentral $F$ or $t$ distributions. 

\section{Proofs} 
Define $l(x)=\log p_{\nu, \lambda}(x)$.  Let us calculate some derivatives in preparation for the main proofs.  By (\ref{dens1}) we have 
\begin{align}
\nonumber
l'(x) &= -\frac{1}{2}+\frac{\nu-2}{4x} +\frac{\sqrt{\lambda}}{2\sqrt{x}} \frac{I'_{(\nu-2)/2}(\sqrt{\lambda x})}{I_{(\nu-2)/2}(\sqrt{\lambda x})}\\
\label{lprime}
&= -\frac{1}{2}+\frac{\nu-2}{2x} + \frac{\sqrt{\lambda}}{2\sqrt{x}} r_{\nu/2}(\sqrt{\lambda x})
\end{align}
where (\ref{lprime}) follows from (\ref{r}) and the formula (Abramowitz and Stegun 1972, (9.6.26))
\begin{equation*}
I_{\nu-1}'(x) = I_\nu(x) + \frac{\nu-1}{x} I_{\nu-1}(x).
\end{equation*}
We differentiate (\ref{lprime}) and get 
\begin{align}
\nonumber
l''(x) &= \frac{2-\nu}{2x^2} + \frac{\lambda}{4x} r'_{\nu/2}(\sqrt{\lambda x}) - \frac{\sqrt{\lambda}}{4x^{3/2}} r_{\nu/2}(\sqrt{\lambda x})\\
\label{lprime2}
&=\frac{2-\nu}{2x^2} + \frac{\lambda}{4x} -\frac{\nu\sqrt{\lambda}}{4x^{3/2}} r_{\nu/2}(\sqrt{\lambda x}) - \frac{\lambda}{4x} r^2_{\nu/2}(\sqrt{\lambda x})\\
\label{lprime22}
&= \frac{\lambda +\nu-4}{4x} -\frac{1}{4} - l'(x) \left(l'(x)+1-\frac{\nu-4}{2x}\right)
\end{align}
where (\ref{lprime2}) holds by the identity
(see, e.g., Amos 1974) 
\begin{equation}
\label{rprime}
r'_\nu(x) = 1-\frac{2\nu-1}{x} r_\nu(x) - r^2_\nu(x),
\end{equation}
and (\ref{lprime22}) holds by (\ref{lprime}).  
Differentiating (\ref{lprime22}) yields
\begin{equation}
\label{lprime3}
l'''(x) = -\frac{1}{4x^2}\left(\lambda+\nu-4 + 2(\nu-4) l'(x)\right) -l''(x)\left(2l'(x)+1-\frac{\nu-4}{2x}\right).
\end{equation}

We also use the following results, which are easily derived from modified Bessel function asymptotics (see, e.g., Abramowitz and Stegun 1972, (9.7.1)). 
\begin{align}
\label{rsmall}
r_{\nu/2}(x) &= \frac{x}{\nu} -\frac{x^3}{\nu^2(\nu+2)} + o(x^3),\quad {\rm as}\ x\downarrow 0;\\
\label{rasymp}
r_{\nu/2}(x) &= 1 - \frac{\nu-1}{2x} + o(x^{-1}),\quad {\rm as}\ x\to \infty.
\end{align}

\begin{proof}[Proof of Lemma~\ref{lem1}]
Let $\nu\in (0, 2)$.  As $\lambda\downarrow (4-\nu)$, we have $g_\nu(\lambda)\to -\infty$ by (\ref{rsmall}); as $\lambda\to\infty$ we have $\lambda g_\nu(\lambda)\to 1/2$ by (\ref{rasymp}).  Thus $g_\nu(\lambda)$ changes signs at least once, from negative to positive.  For any $\lambda_*$ such that $g_{\nu}(\lambda_*)=0$, we have $r_{\nu/2}(t)=(\lambda_* -2)/t$, where 
$t\equiv\sqrt{\lambda_*(\lambda_*+\nu-4)}.$  A calculation using (\ref{rprime}) yields 
$$r'_{\nu/2}(t) = \frac{\lambda_* +2\nu-6}{t^2}.$$
Thus 
\begin{align*}
\frac{\partial g_\nu(\lambda_*)}{\partial \lambda} &= r'_{\nu/2}(t) \left(\frac{2\lambda_* +\nu-4}{2 t}\right) -\frac{1}{t} +(\lambda_* -2) \left(\frac{2\lambda_* +\nu-4}{2 t^3}\right)\\
&=\frac{\sqrt{\lambda_* +\nu-4}}{\lambda_*^{3/2}} > 0.
\end{align*}
Hence every sign change of $g_{\nu}(\lambda)$ is from $-$ to $+$.  There cannot be two or more sign changes, because one of them has to be from $+$ to $-$. 
\end{proof}

\begin{proof}[Proof of Proposition~\ref{prop1}]
We only prove $\lim_{\nu\uparrow 2} \lambda_\nu=2$; a similar argument works for $\lim_{\nu\downarrow 0} \lambda_\nu = 4$. 
For fixed $\lambda >2$ we have 
\begin{equation}
\label{limg}
\lim_{\nu\uparrow 2} g_\nu(\lambda) =r_1(\sqrt{\lambda(\lambda-2)}) - \sqrt{\frac{\lambda-2}{\lambda}}.
\end{equation}
By (\ref{rsmall}), the right hand side of (\ref{limg}) is positive for all sufficiently small $\lambda>2$.  For such $\lambda$ we have $g_\nu(\lambda)>0$ if $\nu$ is sufficiently close to 2, and $\lambda_\nu < \lambda$ by Lemma~\ref{lem1}.  That is, $\limsup_{\nu\uparrow 2} \lambda_\nu \leq \lambda$.  The claim follows by choosing $\lambda \downarrow 2$.  (Obviously $\lambda_\nu>2$ and hence $\liminf_{\nu\uparrow 2} \lambda_\nu\geq 2$.) 
\end{proof}

\begin{proof}[Proof of Theorem~\ref{main}]
We only prove the ``if'' parts because the ``only if'' parts hold after examining certain boundary cases and, if necessary, invoking the ``if'' parts.  For example, to show the ``only if'' statement in part (iii), note that if $\nu>2$, or $\nu=2$ and $\lambda > \lambda_\nu=2$, then by (\ref{lprime}) we have $l'(x) > 0 $ for small $x>0$.  Hence the density cannot be decreasing. 

Part (i) can be proved by showing that $r_{\nu/2}(x)/x$ decreases in $x \in (0,\infty)$ (see Saxena and Alam 1982 or Yu 2011).  By (\ref{lprime}), if $\nu\geq 2$ then $l'(x)$ decreases in $x$, as required. 

To prove part (ii) let us assume $0<\nu<2$ and $\lambda>0$.  As $x\downarrow 0$ we have, by (\ref{rsmall}) and (\ref{lprime2}), $x^2 l''(x)\to (2-\nu)/2 >0$.  As $x\to \infty$ we use (\ref{rasymp}) and (\ref{lprime2}) to obtain $x^{3/2} l''(x)\to -\sqrt{\lambda}/4 <0$.  Thus $l''(x)$ crosses zero at least once from above.  If it crosses zero more than once, then there exists $x_*$ such that $l''(x_*)=0$ and $l'''(x_*)\geq 0$.  By (\ref{lprime3}) we have
$$l'(x_*)\geq \frac{\lambda+\nu-4}{2(4-\nu)}.$$
Putting this in (\ref{lprime22}) we get 
\begin{align*}
l''(x_*) \leq &\frac{\lambda+\nu-4}{4x_*}-\frac{1}{4} - l'(x_*)(l'(x_*)+1) + \frac{\lambda+\nu-4}{2(4-\nu)} \left(\frac{\nu-4}{2x_*}\right)\\
= & - \left(l'(x_*)+\frac{1}{2}\right)^2\leq 0.
\end{align*}
Thus equalities must hold and we have 
$$\frac{\lambda+\nu-4}{2(4-\nu)}=l'(x_*)=-\frac{1}{2},$$ 
which yields $\lambda=0$, a contradiction.  It follows that $l''(x)$ has exactly one sign change on $(0,\infty)$, and at the change point $l''' < 0$, yielding a strict $(+,\, -)$ sign pattern. 

Let us now assume $0<\nu< 2$ and $\lambda< \lambda_\nu$ and prove the ``if'' part of (iii).  The argument for $\nu=2$ is similar, and the case $\lambda=\lambda_\nu$ follows by taking a limit.  As $x\downarrow 0$ we have $l'(x)\to -\infty$; 
as $x\to \infty$ we have $l'(x)\to -1/2$ by (\ref{lprime}).  If $l'(x)$ does become positive, then there exists $0<x_1\leq x_2$ (entry and exit points) such that $l'(x_1)=l'(x_2)=0$ and $l''(x_1)\geq 0\geq l''(x_2)$.  
By (\ref{lprime22}) we have $x_1 \leq \lambda+\nu-4\equiv x_0\leq x_2$, which implies $x_0>0$.  By Lemma~\ref{lem1}
\begin{equation}
\label{eqn1}
l'(x_0)=\frac{\sqrt{\lambda}}{2\sqrt{x_0}} g_\nu(\lambda)< 0.
\end{equation}
We can rule out $l''(x_2)=0$ because it implies $x_2=x_0$ by (\ref{lprime22}) which contradicts (\ref{eqn1}).  Because $l'(x_0)<0,\ l'(x_2)=0,\ l''(x_2)< 0,$ there exists $x^*\in (x_0, x_2)$ such that $l'(x^*)=0$ and $l''(x^*)\geq 0$.  By (\ref{lprime22}), however, we obtain $x^* \leq x_0$, which is a contradiction. 

Finally, suppose $0<\nu<2$ and $\lambda > \lambda_\nu$.  Then $l(x)\to \infty$ and $l'(x)\to -\infty$ as $x\downarrow 0$.  So one mode is at zero.  In view of part (ii), let $\tilde{x}\in (0,\infty)$ be the unique solution of $l''(x)=0$, so that $l''(x)> 0$ or $<0$ according as $x<\tilde{x}$ or $> \tilde{x}$.  If $l'(\tilde{x})\leq 0$ then by (\ref{lprime22}) we have $\tilde{x}\geq \lambda+\nu-4\equiv x_0$.  (We always have $l'(x)+1-(\nu-4)/(2x)>0$ in (\ref{lprime22}) by (\ref{lprime}).) 
By Lemma~\ref{lem1} 
$$l'(\tilde{x})\geq l'(x_0)= \frac{\sqrt{\lambda}}{2\sqrt{x_0}} g_\nu(\lambda) > 0,$$
which is a contradiction.  Thus $l'(\tilde{x})>0$, implying that $l(x)$ is convex with an interior minimum in $(0, \tilde{x})$ and concave with an interior maximum (the second mode) in $(\tilde{x},\infty)$.  This proves part (iv). 
\end{proof}

\begin{proof}[Proof of Proposition~\ref{monomode}]
Existence of the interior mode follows from Corollary~\ref{coro1}.  To show monotonicity of $M(\nu, \lambda)$, we prove that $h(x)\equiv xr_\nu(x)$ increases in $x\in (0,\infty)$ for all $\nu>0$. This implies that $l'(x)$ as given by (\ref{lprime}) increases in $\lambda$, which quickly yields the claim.  We have $h'(x) =r_\nu(x) + x r'_\nu(x)$.  Suppose $h'(x_*)=0$ for some $x_*>0$.  Then $r'_\nu(x_*)=-r_\nu(x_*)/x_*$ and, by (\ref{rprime}), $r^2_\nu(x_*)+(2\nu-2)r_\nu(x_*)/x_* -1 = 0.$  We get 
\begin{align*}
h''(x_*) &=1-(2\nu-2)r'_\nu(x_*) - 2x_* r_\nu(x_*) r'_\nu(x_*)-r^2_\nu(x_*)\\
& = 2
\end{align*}
after simple algebra.  Hence any sign change of $h'(x)$ is from negative to positive.  However, as $x\downarrow 0$ we have 
$h'(x)/x \to 1/\nu>0$.  This precludes possible sign changes.  Hence $h'(x)>0$, as required.
\end{proof}

\begin{proof}[Proof of Proposition~\ref{prop2}]
Assume an interior mode exists.  We refer to $\{(\nu, \lambda):\ \nu>2\ {\rm or}\ \nu=2,\ \lambda>2\}$ as ``the log-concave case'' and $\{(\nu, \lambda):\ 0<\nu<2,\ \lambda>\lambda_\nu\}$ as the bi-modal case.  Denote $x_0=\lambda+\nu-4$.  If $x_0\leq 0$ then (\ref{loose}) is trivial.  If $x_0>0$ and $l'(x_0)>0$ then 
$x_0< M(\nu,\lambda)$ because the density declines after $M(\nu,\lambda)$.  Suppose $x_0>0$ and $l'(x_0)\leq 0$.  By (\ref{lprime22}), $l''(x_0)\geq 0$, implying a bi-modal case, and $x_0$ belongs to the declining phase in the log-convex region.  By Theorem~\ref{main} we have $x_0< M(\nu, \lambda)$.  Thus (\ref{loose}) holds. 

It remains to prove (\ref{bound2}) and (\ref{bound3}) because Sen (1989) has established (\ref{bound1}).
Denote $z=\lambda+\nu-3$.  In the case of $\nu>3$ we need to show $l'(z)>0$.  In the bi-modal case we need $l'(z)<0$ instead.  This suffices because, letting $m$ be the local minimum of $l(x)$, we have $m\leq x_0\equiv \lambda+\nu-4,$ which follows from $l'(m)=0,\ l''(m)\geq 0$ and (\ref{lprime22}).  Hence $m < z$, and (\ref{bound3}) holds if $l'(z) < 0$.  By (\ref{lprime}), 
$$2\sqrt{\frac{z}{\lambda}} l'(z) = r_{\nu/2}(\sqrt{\lambda z}) - \frac{\lambda -1}{\sqrt{\lambda z}}\equiv h(\lambda).$$
Let us follow the proofs of Lemma~\ref{lem1} and Proposition~\ref{monomode} in analyzing the sign pattern of $h(\lambda)$. 

Suppose $0< \nu <2$ and $\lambda> 4-\nu$.  As $\lambda \downarrow (4-\nu)$ we get 
\begin{equation}
\label{nu}
h(\lambda) \to r_{\nu/2}(\sqrt{4-\nu})-\frac{3-\nu}{\sqrt{4-\nu}}. 
\end{equation}
One can verify numerically that the right hand side of (\ref{nu}) is negative for all $0<\nu<2$ (there is only one variable over a small range).  Suppose there exists some $\lambda_*\in (4-\nu, \infty)$ such that $h(\lambda_*)=0$.  Then $r_{\nu/2}(\sqrt{\lambda_* z_*})=(\lambda_* -1)/\sqrt{\lambda_* z_*}$ where $z_*=\lambda_*+\nu-3$, and by (\ref{rprime}), $r'_{\nu/2}(\sqrt{\lambda_* z_*}) =  (\nu-2)/(\lambda_* z_*)$.  Some algebra yields 
\begin{align}
\label{hprime1}
h'(\lambda_*) =& r'_{\nu/2}(\sqrt{\lambda_* z_*}) \frac{2\lambda_* +\nu -3}{2\sqrt{\lambda_* z_*}} +\frac{(1-\nu)\lambda_* + 3-\nu}{2(\lambda_* z_*)^{3/2}} \\
\label{hprime2}
=& \frac{\nu-3}{2\sqrt{\lambda_*^3 z_*}} <0. 
\end{align}
Thus any sign change of $h(\lambda)$ is from $+$ to $-$.  Since $h(\lambda)$ is negative as $\lambda\downarrow (4-\nu)$, there cannot be any sign change.  This proves $l'(z)< 0$ for $\lambda>\lambda_\nu> 4-\nu$. 

Finally, let us assume $\nu>3$.  The calculation (\ref{hprime1})--(\ref{hprime2}) now shows that any sign change of $h(\lambda)$ is from $-$ to $+$.  But $h(\lambda) \to \infty$ as $\lambda \downarrow 0$.  Thus sign changes are impossible and we conclude that $l'(z)>0$, as required. 
\end{proof}

{\bf Remark.}  To prove an inequality $h(\lambda)> 0$ for all $\lambda$, we show that (i) it holds for small (respectively, large) $\lambda$; (ii) $h'(\lambda)$ is positive (respectively, negative) assuming $h(\lambda)=0$.  This simple method works well for functions involving $r_\nu(x)$ and seems applicable to other problems. 


\end{document}